\documentclass{article}
\usepackage{hyperref}
\usepackage{algorithm}
\usepackage{algpseudocode}
\usepackage{amsfonts}
\usepackage{caption}
\usepackage{subcaption}
\usepackage{amsmath}
\DeclareMathOperator*{\argmin}{argmin} 
\newcommand{\Span}{\mathrm{span}}
\newcommand{\dimension}{\mathrm{dim}}
\newcommand{\bt}{\boldsymbol{t}}
\newcommand{\bp}{\boldsymbol{p}}
\newcommand{\bc}{\boldsymbol{c}}
\newcommand{\bx}{\boldsymbol{x}}
\newcommand{\tequi}{\boldsymbol{t_{\mathsf{equi}}}}

\usepackage{amsthm}
\usepackage{graphicx}
\newtheorem{theorem}{Theorem}
\newtheorem{lemma}[theorem]{Lemma}
\newtheorem{assumption}[theorem]{Assumption}
\newtheorem{corollary}[theorem]{Corollary}
\usepackage{setspace}
\usepackage{easy-todo}

\newcommand{\R}{\mathbb R}
\newcommand{\PP}{\mathbb P}

\usepackage{ifthen}
\newboolean{ShowRevisions}
\setboolean{ShowRevisions}{true}

\usepackage[normalem]{ulem}
\ifthenelse{\boolean{ShowRevisions}}
{	
	
	\newcommand{\comment}[1]{\textcolor{purple}{#1}}
}
{
	
	\newcommand{\comment}[1]{}	
}

\usepackage[backend=biber, style=numeric-comp , url=false]{biblatex}
\usepackage{lineno}

\DeclareSourcemap{
  \maps[datatype=bibtex]{
    \map{
      \step[fieldset=issn, null]
      \step[fieldset=isbn, null]
    }
  }
}
\title{Subspace Acceleration for a Sequence of Linear Systems and Application to Plasma Simulation}

\author{Margherita Guido\thanks{Ecole Polytechnique Fédérale de Lausanne (EPFL), Institute of Mathematics, 1015 Lausanne, Switzerland (margherita.guido@epfl.ch, daniel.kressner@epfl.ch)}\,\,\thanks{Ecole Polytechnique Fédérale de Lausanne (EPFL), Swiss Plasma Center (SPC), 1015 Lausanne, Switzerland (paolo.ricci@epfl.ch)}  \and Daniel Kressner\footnotemark[1] \and Paolo Ricci\footnotemark[2] }

\begin{document}

\maketitle

\begin{abstract}

We present an acceleration method for sequences of large-scale linear systems, such as the ones arising from the numerical solution of time-dependent partial differential equations coupled with algebraic constraints. We discuss different approaches to leverage the subspace containing the history of solutions computed at previous time steps in order to generate a good initial guess for the iterative solver. In particular, we propose a novel combination of reduced-order projection with randomized linear algebra techniques, which drastically reduces the number of iterations needed for convergence.
We analyze the accuracy of the initial guess produced by the reduced-order projection when the coefficients of the linear system depend analytically on time. Extending extrapolation results by Demanet and Townsend to a vector-valued setting, we show that the accuracy improves rapidly as the size of the history increases, a theoretical result confirmed  by our numerical observations. 
In particular, we apply the developed method to the simulation of plasma turbulence in the boundary of a fusion device, showing that the time needed for solving the linear systems is significantly reduced.

\end{abstract}

\section{Introduction}
\label{sec:intro}

The numerical solution of time-dependent partial differential equations (PDEs) often leads to sequences of linear systems of the form 
\begin{equation}
    \label{eq:seq}
    A(t_{i}) \boldsymbol{x}(t_{i}) = \boldsymbol{b}(t_{i}) \quad \quad i=0,1,2 \cdots,
\end{equation}
where $t_{0}< t_1 <t_2 < \cdots$ is a discretization of time $t$, and both the system matrix $A(t_{i}) \in \mathbb{R}^{n \times n}$ and the right-hand side $\boldsymbol{b}(t_{i}) \in \mathbb{R}^n$ depend on time. Typically, the systems~\eqref{eq:seq} are available only consecutively. 
Such sequences of linear systems can arise in a number of applications, including implicit time stepping schemes for the solution of PDEs or iterative solutions of non-linear equations and optimization problems. 
A relevant example is given by time-dependent PDEs solved in  presence of algebraic constraints. In this case, even when an explicit time stepping method is used to evolve the nonlinear PDE, the discretization of  the algebraic constraints leads to linear systems that need to be solved at every (sub-)timestep. 
This  is the case of the simulation of  turbulent plasma dynamics ~\cite{Fasoli2016}, where a linear constraint (Maxwell equations) is imposed upon the plasma dynamics described by a set of non linear fluid or kinetics equations. The  linear systems resulting from the discretized algebraic constraints may feature millions of degrees of freedom, hence their solution is often computationally very expensive. 

 One usually expects that the linear system~\eqref{eq:seq} changes slowly in subsequent time steps.
 This work is focused on exploiting this property to accelerate iterative solvers, such as
CG~\cite{Hestenes1952} for symmetric positive definite matrices and GMRES \cite{Saad1986} for general matrices. An obvious way to do so is to supply the iterative solver for the timestep $t_{i+1}$ with the solution of~\eqref{eq:seq} at timestep $t_{i}$, as initial guess. As a more advanced technique, in the context of Krylov subspace methods, \textit{subspace recycling methods}~\cite{Soodhalter2020} such as GCROT~\cite{Mathematics2012} and GMRES-DR~\cite{Morgan2000} have been proposed. Such methods have been developed in the case of a single linear system, to enrich the information when restarting the iterative solver. The idea behind is often to accelerate the convergence by suppressing parts of the  spectrum of the matrix, including the corresponding approximate invariant subspace in the Krylov minimization subspace.
GCROT and GMRES-DR have then been adapted to sequences of linear systems in~\cite{Parks2006a}, recycling selected subspaces from one system to the next. 
For this class of methods to be efficient, it is necessary that the sequence of matrices undergoes local changes only, that is, the difference $A(t_{i+1})- A(t_{i})$ is computationally  cheap to apply. For example, one can expect this difference matrix to be sparse when time dependence is restricted to a small part of the computational domain, e.g., through time-dependent boundary conditions. We refer to~\cite{Soodhalter2020} for a more complete survey of subspace recycling methods and their applications.
In \cite{Carlberg2016}, subspace recycling was combined with goal-oriented POD (Proper Orthogonal Decomposition) in order to limit the size of the subspaces involved in an augmented CG approach. Simplifications occur when the matrices $A(t_i)$ are actually a fixed matrix $A$ shifted by different scalar multiples of the identity matrix, because Krylov subspaces are invariant under such shifts. In the context of subspace recycling, this property has been exploited in, e.g.,~\cite{Soodhalter2014}, and in~\cite{Soodhalter2016} it is shown how a smoothly varying right-hand side can be incorporated.

When $A(t_i)$ and $\textbf{b}(t_i)$ in~\eqref{eq:seq} are samples of smooth matrix/vector-valued functions, one expects that the subspace of the previously computed solutions contains a very good approximation of the current one. This can be exploited to construct a better initial guess, either explicitly through (polynomial) extrapolation, or implicitly through projection techniques.
Examples of the extrapolation approach include polynomial POD extrapolation~\cite{Grinberg2011}, weighted group extrapolation methods~\cite{Ye2020} and a stabilized, least-squares  polynomial extrapolation method~\cite{Austin2021}, for the case that only the right-hand side evolves in time. For the same setting, projection techniques have been introduced by Fischer~\cite{Fischer1998}. Following this first work, several approaches have been developed to extract an initial guess from the solution of a reduced-order model, constructed from projecting the problem to a low-dimensional subspace spanned by previous solutions.
In~\cite{Tromeur-Dervout2006}, such an approach is applied to fully implicit discretizations of nonlinear evolution problems, while~\cite{Markovinovic2006} applies the same idea to the so called IMPES scheme used for simulating two-phase flows through heterogeneous porous media. 

In this paper, we develop a new projection technique for solving sequences of linear systems that combines projection with randomized linear algebra techniques, leading to considerably reduced cost.  Moreover, a novel convergence analysis of the algorithm is carried out to show its efficiency. This is also proved numerically by applying the algorithm to the numerical simulation of turbulent plasma in the boundary of a fusion device.

The rest of this paper is organized as follows. In Section~\ref{sec:algorithms}, we first discuss general subspace acceleration techniques based on solving a projected linear system and then explain how randomized techniques can be used to speed up existing approaches. In Section~\ref{sec:theory}, a convergence analysis of these subspace acceleration techniques is presented. In Section~\ref{sec:test} we first discuss numerical results for a test case to demonstrate the improvements that can be attained by the new algorithm in a somewhat idealistic setting. In Section~\ref{sec:GBS} our algorithm is applied to large-scale turbulent simulation of plasma in a tokamak, showing a significant reduction of computational time.

\section{Algorithm} \label{sec:algorithms}

The algorithm proposed in this work for accelerating the solution of the sequence of linear systems~\eqref{eq:seq} uses randomized techniques to lower the cost of a POD-based strategy, such as the one proposed in~\cite{Markovinovic2006}.
Recall that we aim at solving the linear systems $A(t_{i}) \boldsymbol{x}(t_{i}) = \boldsymbol{b}(t_{i})$ consecutively for $i = 0,1,\cdots$. We make no assumption on the symmetry of $A(t_{i}) \in \mathbb{R}^{n \times n}$ and thus GMRES is an appropriate choice for solving each linear system.
Supposing that, at the $i$th timestep, $M$ previous solutions are available, we arrange them into the 
history matrix 
\[
 X = \left[\boldsymbol{x}(t_{i-M})\,|\,\cdots \,|\,\boldsymbol{x}(t_{i-1})\right] \in \mathbb R^{n\times M}.
\]
where the notation on the right-hand side indicates the concatenation of columns.
Instead of using the complete history, which may contain redundant information, one usually selects a subspace $\mathcal{S} \subset \Span(X)$ of lower dimension $m\le M$. Then, the initial guess for the $i$th linear system is obtained from choosing the element of $\mathcal{S}$ that minimizes the residual:
\[
\min\limits_{s \in \mathcal S} 
\| A(t_{i})\boldsymbol{s} - \boldsymbol{b}(t_{i})\|_2 = 
\min\limits_{\boldsymbol{z} \in \mathbb{R}^{m}}\| A(t_{i})Q \boldsymbol{z} - \boldsymbol{b}(t_{i})\|_2,
\]
where the columns of $Q \in \mathbb R^{n\times m}$ contain an orthonormal basis of $\mathcal S$. We use $\|\cdot \|_2$ to denote the Euclidean norm for vectors and the spectral norm for matrices.
The described approach is summarized in Algorithm~\ref{alg:1}, which is a template that needs to be completed by an appropriate choice of the subspace $\mathcal{S}$, in Sections~\ref{sec:POD} and~\ref{sec:RRF}.

\begin{algorithm}
\linespread{1.3}\selectfont 
    \caption{Solution of $i$th linear system $ A(t_{i}) \boldsymbol{x}(t_{i}) = \boldsymbol{b}(t_{i})$}
    \label{alg:1}
    \begin{algorithmic}[1]
    \Require History of $M$ solutions $\left\{\boldsymbol{x}(t_{i-M}),\cdots , \boldsymbol{x}(t_{i-1})\right\} $
    
   \State $X = \left[\boldsymbol{x}(t_{i-M})\,|\,\cdots \,|\,\boldsymbol{x}(t_{i-1})\right]$

    \State Generate $Q \gets$ orthonormal basis for $\mathcal S \subseteq \Span(X)$, $\dimension(\mathcal S)=m \leq M$
    \State Compute $\boldsymbol{s}^{\star} =   \argmin\limits_{\boldsymbol{z} \in \mathbb{R}^{m}}\| A(t_{i}) Q \boldsymbol{z} - \boldsymbol{b}(t_{i})\|_2 \in \mathcal{S} $ 
    \State Solve $A(t_{i}) \boldsymbol{x}(t_{i}) = \boldsymbol{b}(t_{i})$ using GMRES with initial guess $\boldsymbol{s}^{\star} \in \mathcal{S}$
    \end{algorithmic}
    \end{algorithm}
    
If the complete history is used, $\mathcal S = \Span(X)$, then computing $Q$ via a QR decomposition~\cite{Golub2013}, as required in Step 2, costs $\mathcal{O}(M^2 n)$ operations.
In addition, setting up the linear least-squares problem in Step 3 of Algorithm~\ref{alg:1} requires $M$ (sparse) matrix-vector products in order to compute $A(t_{i}) Q$. The standard approach for solving the linear least-squares problem proceeds through the QR decomposition of that matrix and costs another $\mathcal{O}(M^3 + M^2 n)$ operations. This strong dependence of the cost on $M$ effectively forces a rather small choice of $M$,  neglecting relevant components of the solutions that could be contained in older solutions only. In the following, we discuss two strategies to overcome this problem.

\subsection{Proper Orthogonal Decomposition}
\label{sec:POD}

An existing strategy~\cite{Markovinovic2006} to arrive at a low-dimensional subspace $\mathcal{S} \subset \Span(X)$ uses a POD approach~\cite{Kunisch1999} and computes the orthonormal basis $Q$ for $\mathcal{S}$ through a truncated SVD (Singular Value Decomposition) of $X$; see Algorithm~\ref{alg:2}. Note that only the first $m$ left singular vectors $\boldsymbol{\Psi}_{1}, \cdots, \boldsymbol{\Psi}_{m}$ need to be computed in Step 2. 

\begin{algorithm}[H]
\linespread{1.3}\selectfont 
    \caption{Method 1 (POD) to generate basis $Q = Q_{\mathsf{POD}}$}\label{alg:2}
    \begin{algorithmic}[1]
        \Require History of $M$ solutions $\left\{\boldsymbol{x}(t_{i-M}) ,\cdots , \boldsymbol{x}(t_{i-1}) \right\}$
             \State $X = \left[\boldsymbol{x}(t_{i-M})\,|\,\cdots \,|\,\boldsymbol{x}(t_{i-1})\right]$
        \State  Compute SVD of $X$:  $\left[ \Psi , \Sigma , \Phi \right] =  \text{svd}(X)$
        
        \State $Q_{\mathsf{POD}} =  \left[ \boldsymbol{\Psi}_{1}   |\,\cdots \,| \boldsymbol{\Psi}_{m}\right] \in \mathbb{R} ^{n \times m}$
    \end{algorithmic}
\end{algorithm}

Thanks to basic properties of the SVD, the basis $Q_{\mathsf{POD}}$ enjoys the following optimality property~\cite{Volkwein2013}:
\begin{equation}
 \label{eq:optpod}
 \| (I - Q_{\mathsf{POD}} Q_{\mathsf{POD}}^{T}) X \|_{F}^{2} = \sum_{k=m+1}^{M} \sigma_{k}^{2} = \min_{Q \in \R^{n\times n} \atop Q^T Q = I_m} \| (I - Q Q^{T}) X \|_{F}^{2},
\end{equation}
where $\|\cdot\|_F$ denotes the Frobenius norm and $\sigma_1 \ge \sigma_2 \ge \cdots \ge \sigma_M \ge 0$ are the singular values of $X$. In words, the choice $Q_{\mathsf{POD}}$ minimizes the error of orthogonally projecting the columns of $X$ onto an $m$--dimensional subspace.
The relation to the singular values of $X$ established in~\eqref{eq:optpod} also allows one to choose $m$ adaptively, e.g., by choosing $m$ such that 
most of the variability in the history matrix $X$ is captured. 

At every time step, the history matrix $X$ gets modified by removing its first column and appending a new last column. The most straightforward implementation of Algorithm~\ref{alg:2} would compute the SVD needed of Step 2 from scratch at every time step, leading to a complexity of $\mathcal O(nM^2)$ operations. In principle, SVD updating techniques, such as the ones presented in~\cite{Brand2006} and~\cite{Chen2023}, could be used to reduce this complexity 
to $O(mn + m^3)$ for every time step. However, in the context of our application, there is no need to update a complete SVD (in particular, the right singular vectors are not needed) and the randomized techniques discussed in the next section seem to be preferable.

\subsection{Randomized range finder}
\label{sec:RRF}
In this section, an alternative to the POD method (Algorithm~\ref{alg:1}) for generating the low-dimensional subspace $\mathcal{S} \subset \Span(X)$ is presented, relying on randomized techniques. The randomized SVD from~\cite{Halko2011} applied to the $n\times M$ history matrix $X$ proceeds as follows. First, we draw an $M\times m$ Gaussian random matrix $Z$, that is, the entries of $Z$ are independent and identically distributed (i.i.d) standard normal variables. Then the so-called \emph{sketch}
\begin{equation} \label{eq:currenttimestep}
\Omega = X Z = \left[ \boldsymbol{x}(t_{i-M})\,|\, \cdots \,|\, \boldsymbol{x}(t_{i-1}) \right] Z
\end{equation}
is computed, followed by a reduced QR decomposition $\Omega = QR$. This only involves the $n\times m$ matrix $\Omega$, which for $m\ll M$ is a significant advantage compared to Algorithm~\ref{alg:2}, which requires the SVD of an $n\times M$ matrix. The described procedure is contained in lines 2--4 and 11 of Algorithm~\ref{alg:4} below.

According to~\cite[Theorem~10.5]{Halko2011}, the expected value (with respect to $Z$) of the error returned by the randomized SVD satisfies
    \begin{equation} \label{eq:errorrandsvd}
        \mathbb{E} \| (I- QQ^{T})X \|_{F} \leq \Big(1+ \frac{r}{p-1}\Big)^{1/2} \Big(\sum_{k > r} \sigma_{k}^{2}\Big)^{1/2}, 
    \end{equation}
where we partition $m = r + p$ for a small oversampling parameter $p\ge 2$. Also, the tail bound from~\cite[Theorem~10.7]{Halko2011} implies that it is highly unlikely that the error is much larger than the upper bound~\eqref{eq:errorrandsvd}. Comparing~\eqref{eq:errorrandsvd} with the error~\eqref{eq:optpod}, we see that the randomized method is only a factor $\sqrt{2}$ worse than the optimal basis of roughly \emph{half} the size produced by POD. As we  also see in our experiments of Section~\ref{sec:test}, this bound is quite pessimistic and usually the randomized SVD performs nearly as good as POD using bases of the \emph{same} size.

\begin{algorithm}[H]
    \linespread{1.3}\selectfont 
        \caption{Method 2 (Randomized Range Finder) to generate basis $Q$}\label{alg:4} 
        \begin{algorithmic}[1]
    \Require{History of $M$ solutions $\left\{\boldsymbol{x}(t_{i-M}) ,\cdots , \boldsymbol{x}(t_{i-1}) \right\}$.  }
    \Statex Optional: $\boldsymbol{x}(t_{i-M-1})$, matrices $\Omega$ and $Z$ from previous time step (see~\eqref{eq:previoustimestep})
            \If{$\Omega$ is computed from scratch}
               \State $X = \left[\boldsymbol{x}(t_{i-M})  |\,\cdots \,| \boldsymbol{x}(t_{i-1})\right] \in \mathbb{R} ^{n \times M}$
                \State Draw Gaussian random matrix $Z = \left[\boldsymbol{z}_{1} |\,\cdots \,| \boldsymbol{z}_{M} \right]^{T} \in \mathbb{R}^{M  \times m }$
                \State $ \Omega = XZ \in \mathbb{R}^{n  \times m }$
            \Else 
                \State $\Omega  =  \Omega - \boldsymbol{x}(t_{i-M-1})\boldsymbol{z}_1$  
                \Comment{$\Omega$ \textit{is updated} }
            \State $ \boldsymbol{z}_k = \boldsymbol{z}_{k+1} \quad k = 1, \cdots, M-1$
            \State Draw new Gaussian random vector $\boldsymbol{z}_M \in \mathbb{R}^{m}$
                \State $\Omega  =   \Omega +\boldsymbol{x}(t_{i-1})\boldsymbol{z}_M^{T}$ \Comment{$\Omega$ \textit{is updated} }
            \EndIf
            \State $\left[ Q, R \right] = $reduced QR of $\Omega$
        \end{algorithmic}
    \end{algorithm}
Instead of performing the randomized SVD from scratch in every timestep, one can easily exploit the fact that only a small part of the history matrix is modified. To see this, let us consider the sketch from the previous timestep:
\begin{equation} \label{eq:previoustimestep}
\Omega_{\mathsf{prev}} = \left[ \boldsymbol{x}(t_{i-M-1})\,|\, \cdots \,|\, \boldsymbol{x}(t_{i-2}) \right] Z^{\mathsf{prev}}.
\end{equation}
Comparing with~\eqref{eq:currenttimestep}, we see that the sketch $\Omega$ of the current timestep is obtained by removing the contribution from the solution $\boldsymbol{x}(t_{i-M-1})$ and adding the contribution of $\boldsymbol{x}(t_{i-1})$. The removal is accomplished in line 6 of Algorithm~\ref{alg:4} by a rank-one update:
\[
 \Omega_{\mathsf{prev}} - \boldsymbol{x}(t_{i-M-1}) \boldsymbol{z}^{\mathsf{prev}}_{1}=  \left[ \boldsymbol{0}\,|\,\boldsymbol{x}(t_{i-M})\,|\,   \cdots \,|\, \boldsymbol{x}(t_{i-2}) \right] Z^{\mathsf{prev}}.
\]
By a cyclic permutation, we can move the zero column to the last column,
$\left[ \boldsymbol{x}(t_{i-M})\,|\,   \cdots \,|\, \boldsymbol{x}(t_{i-2}) \,|\, \boldsymbol{ 0}\right]$, updating $Z$ as in line 7 of Algorithm~\ref{alg:4}. Finally, the contribution of the latest solution is incorporated by adding the rank-one matrix $\boldsymbol{x}(t_{i-1}) \boldsymbol{z}_M^T$, where $\boldsymbol{z}_M$ $ \in \R^m$ is a newly generated Gaussian random vector that is stored in the last row of $Z$. Under the (idealistic) assumption that all solutions are exactly computed (and hence deterministic), the described progressive updating procedure is mathematically equivalent to computing the randomized SVD from scratch. In particular, the error bound~\eqref{eq:errorrandsvd} continues to hold.

Lines 6--9 of Algorithm~\ref{alg:4} require $\mathcal O(nm)$ operations. When using standard updating procedures for QR decomposition~\cite{Golub2013}, line 11 has the same complexity. This compares favorably with the $\mathcal O(nM^2)$ operations needed by Algorithm~\ref{alg:2} per timestep.

When performing the progressive update of $\Omega$ over many timesteps, one can encounter numerical issues due to numerical cancellation in the repeated subtraction and addition of contributions to the sketch matrix. To avoid this, the progressive update is carried out only for a fixed number of timesteps, after which a new random matrix $Z$ is periodically generated and $\Omega$ is computed from scratch. 

\section{Convergence Analysis}
\label{sec:theory}

We start our convergence analysis of the algorithms from the preceding section by considering analytical properties of the history matrix $X = [\boldsymbol{x}(t_{i-M})\,|\,\cdots \,|\,\boldsymbol{x}(t_{i-1})]$. After reparametrization, we may assume without loss of generality that each of the past timesteps is contained in the interval $[-1,1]$:
\[
 -1 = t_{i-M} < \cdots < t_{i-1} = 1.
\]
For notational convenience, we define 
\begin{equation} \label{def:X}
  X \equiv X(\bt) := \left[ \boldsymbol{x}(t_{i-M})\,|\, \cdots\,|\,\boldsymbol{x}(t_{i-1}) \right], \quad \bt = \left[
  t_{i-M}, \cdots,t_{i-1} \right],
\end{equation}
where $\boldsymbol{x}(t)$ satisfies the (parametrized) linear system
\begin{equation}
    A(t) \boldsymbol{x}(t) = \boldsymbol{b}(t), \quad A:[-1,1] \to \R^{n \times n}, \quad \boldsymbol{b} : 
    [-1,1] \to \R^{n},
    \label{eq:parsys}
\end{equation}
that is, each entry of $A$ and $\boldsymbol{b}$ is a scalar function on the interval $[-1,1]$. Indeed, for the convergence analysis, we assume that each linear system of the sequence in~\eqref{eq:seq} is obtained by sampling the parametrized system in~\eqref{eq:parsys} in $t_i \in [-1,1]$. 
In many practical applications, like the one described in Section~\ref{sec:GBS}, the time dependence in~\eqref{eq:parsys} arises from time-dependent coefficients in the underlying PDEs. Frequently, this dependence is real analytic, which prompts us to make the following smoothness assumption on $A$, $\boldsymbol b$.
\begin{assumption} \label{assume:analyticity}
Consider the open Bernstein ellipse $E_{\rho} \subset \mathbb C$ for $\rho > 1$, that is, the open ellipse with foci $\pm 1$ and semi-minor/-major axes summing up to $\rho$.
We assume that $A : \left[  -1, 1  \right] \to \mathbb{C}^{n \times n}$ and $ \boldsymbol{b} : \left[  -1, 1  \right] \to \mathbb{C}^{n}$ admit extensions that are analytic on $E_{\rho}$ and continuous on $\bar{E_{\rho}}$ (the closed Bernstein ellipse), such that $A(t)$ is invertible for all $t\in \bar{E_{\rho}}$. In particular, this implies that $\boldsymbol{x}(t) = A^{-1}(t) \boldsymbol{b}(t)$ is analytic on  $E_{\rho}$ and $\kappa_\rho := \max_{t \in \partial E_{\rho}} \|\boldsymbol{x}(t) \|_2$ is finite.
\end{assumption}

\subsection{Compressibility of the solution time history}

The effectiveness of POD-based algorithms relies on the compressibility of the solution history, that is, the columns of $X$ can be well approximated by an $m$--dimensional subspace with $m \ll M$. According to~\eqref{eq:optpod}, this is equivalent to stating that the singular values of $X$ decrease rapidly to zero.
Indeed, this property is implied by Assumption~\ref{assume:analyticity} as shown by the 
following result, which was stated in~\cite{Kressner2011} in the context of low-rank methods for solving parametrized linear systems. 
\begin{theorem}[\protect{\cite[Theorem~2.4]{Kressner2011}}]
    \label{thm:1}
    Under Assumption~\ref{assume:analyticity}, the $k$th largest singular value $\sigma_k$ of the history matrix $X(\bt)$ from~\eqref{def:X} satisfies
    \[
        \sigma_{k} \leq \frac{2 \rho \kappa_\rho \sqrt{M}}{1 - \rho^{-1}}   \rho^{-k}.
    \]
\end{theorem}

Combined with~\eqref{eq:optpod}, Theorem~\ref{thm:1} implies that
the POD basis $Q_{\mathsf{POD}} \in \mathbb R^{n\times m}$ satisfies the error bound
\[
    \| (I - Q_{\mathsf{POD}} Q_{\mathsf{POD}}^{T}) X \|_{F}^{2} \leq \frac{4 \rho^{2} \kappa_\rho^{2} M}{(1 - \rho^{-1})^{2}}   (\rho^{-(m+1)} - \rho^{-(M+1)}).
\]

\subsection{Quality of prediction without compression}

Algorithm~\ref{alg:1} determines the initial guess $\boldsymbol{s}^*$ for the next time step $t_{i} > t_{i-1} = 1$  by solving the minimization problem 
\begin{equation} \label{eq:minrecall} 
 \boldsymbol{s}^* = \argmin_{\boldsymbol{s}\in \mathcal S} \|A(t_{i}) \boldsymbol{s} - \boldsymbol{b}(t_{i})\|_2.
\end{equation}
In this section, we will assume, additionally to Assumption~\eqref{assume:analyticity}, that $\mathcal S = \Span(X(\bt))$, that is, $X(\bt)$ is not compressed. Our analysis focuses on uniform timesteps $\tequi = \left[ t_{i-M},\cdots, t_{i-1}\right]$ defined by
\[
 t_{i-M} = -1,\ t_{i-M+1} = -1+\Delta t,\ \cdots,\ t_{i-2} = 1-\Delta t,\ t_{i-1} = 1, \quad \Delta t  = 2/(M-1).
\]
Note that the next timestep $t_{i} = 1 + \Delta t$ satisfies $t_{i} \in E_\rho$ if and only if $\rho > t_{i} + \sqrt{t_{i}^2-1} \approx 1 + \sqrt{2 \Delta t}$. The following result shows how the quality of the initial guess rapidly improves (at a square root exponential rate, compared to the exponential rate of Theorem~\ref{thm:1}) as $M$, the number of previous time steps in the history, increases. 

\begin{theorem} \label{theorem:main}
Under Assumption~\eqref{assume:analyticity}, the initial guess constructed by Algorithm~\ref{alg:1} with $\mathcal S = \Span(X)$ satisfies the error bound
    \[
        \| A(t_{i}) \boldsymbol{s}^{*} - \boldsymbol{b}(t_{i})\|_2 \leq 2\| A(t_{i})\|_2 \kappa_\rho    \Big[  \frac{1}{1-r}+ 
            \frac{C(M,R) \rho}{ (\rho -1)\sqrt{\rho^2 r^2 -1}} \Big] r^{R+1},
    \]
    with $C(M,R) = 5 \sqrt{5} \sqrt{2R+1} \sqrt{M} / \sqrt{2(M-1)}$, for any $R \leq \frac{1}{2} \sqrt{M-1}$, $r = (t_{i} + \sqrt{t_{i}^{2} -1})/ \rho < 1$.
    \label{thm:res}
\end{theorem}

\subsubsection{Proof of Theorem~\ref{theorem:main}}
\label{sec:proofthm}
The rest of this section is concerned with the proof of Theorem~\ref{theorem:main}. We establish the result by making a connection to vector-valued polynomial extrapolation and extending results by Demanet and  Townsend~\cite{Demanet2019a} on polynomial extrapolation to the vector-valued setting.

Let $\PP_{R} \subset \R^n[t]$ denote the subspace of vector-valued polynomials of length $n$ and degree at most $R$ for some $R \le M-1$. We recall that any $\textbf{v}  \in \PP_{R}$ takes the form $\textbf{v}(t) = \textbf{v}_0 +  \textbf{v}_1 t + \cdots + \textbf{v}_R t^R$ for constant vectors $\textbf{v}_0, \cdots,  \textbf{v}_R \in \R^n$. Equivalently, each entry of $\textbf{v}$ is a (scalar) polynomial of degree at most $R$.
In our analysis we consider vector-valued polynomials of the particular form
\begin{equation} \label{eq:particularp}
\boldsymbol{p}(t)= X(\tequi) \boldsymbol{y}(t),
\end{equation}
for a vector-valued polynomial $\boldsymbol{y}(t)$ of length $M$. A key observation is that the evaluation of $\boldsymbol{p}$ in the next timestep $t_{i}$ satisfies $\boldsymbol{p}(t_{i}) \in \Span(X(\tequi)) = \mathcal{S}$. According to~\eqref{eq:minrecall}, $\mathbf{s^*}$ minimizes the residual over $\mathcal{S}$. Hence, the residual can only increase when we replace $\mathbf{s^*}$ by $\boldsymbol{p}(t_{i})$ in
\begin{align}
\| A(t_{i}) \boldsymbol{s^{*}} - \boldsymbol{b}(t_{i}) \|_2 & 
    \leq   \| A(t_{i}) \boldsymbol{p}(t_{i}) - \boldsymbol{b}(t_{i})     \|_2 \nonumber \\
    &  \leq 
    \| A(t_{i})\|_2  \| \boldsymbol{p}(t_{i})- \boldsymbol{x} (t_{i})    \|_2.
    \label{eq:boundres}
\end{align}
Thus, it remains to find a polynomial of the form~\eqref{eq:particularp} for which we can establish convergence of the extrapolation error $ \| \boldsymbol{p}(t_{i})- \boldsymbol{x} (t_{i})    \|_2$. For this purpose, we will choose $\boldsymbol{p}_R \in \PP_{R}$ to be the least-squares approximation of the $M$ function samples contained in $X(\tequi)$: 
    \begin{equation}
              \boldsymbol{p}_{R} :=  \argmin_{\boldsymbol{p} \in \PP_{R}}\| X(\tequi) - P(\tequi)  \|_{F},
              \quad P(\tequi) = \left[ \bp(t_{i-M})\,|\, \cdots\,|\,\bp(t_{i-1}) \right].
              \label{eq:min}
    \end{equation}
    We will represent the entries of $\boldsymbol{p}_{R}$ in the Chebyshev polynomial basis:
    \begin{equation}
        \boldsymbol{p}_{R}(t) = q_{0}(t) \boldsymbol{c}_{0,p} + 
        q_{1}(t) \boldsymbol{c}_{1,p} + \cdots + q_{R}(t) \boldsymbol{c}_{R,p},
        \label{eq:pol}
    \end{equation}
    where $\boldsymbol{c}_{k,p} \in \R^n$ and $q_{k}$ denotes the 
    Chebyshev polynomial of degree $k$, that is,    
    $q_{k}(t) = \cos (k \cos^{-1}t)$ for $t\in [-1,1]$. Setting
     \begin{equation} \label{eq:chebpolynomial}
       C_{p} = \left[ \boldsymbol{c}_{0,p}|  \cdots |  \boldsymbol{c}_{R,p} \right]
 \in \mathbb{R}^{n \times (R+1)}, \quad 
 \boldsymbol{q}_{R}(t) = \left[ q_{0}(t), \cdots, q_{R}(t)  \right]^{T},
          \end{equation}
     we can express~\eqref{eq:pol} more compactly as $\boldsymbol{p}_{R}(t) = C_{p} \boldsymbol{q}_{R}(t)$. Thus,
     \[
      P_R(\tequi) = C_{p} Q_R(\tequi), \quad Q_R(\tequi) = \
    \left[ \boldsymbol{q}_{R}(t_1)| \cdots |\boldsymbol{q}_{R}(t_M) \right].
     \]
     In view of~\eqref{eq:min}, the matrix of coefficients
     $C_{p}$ is determined by minimizing $\| X(\tequi) - C_{p} Q_R(\tequi)\|_{F}$.
     Because $R \le M-1$, the matrix $Q_R(\tequi)$ has full row rank and thus the solution of this least-squares problem is given by $C_{p}  = X(\tequi)Q_{R}(\tequi)^{\dagger}$ with $Q_{R}(\tequi)^{\dagger} = Q_{R}(\tequi)^{T} (Q_{R}(\tequi) Q_{R}(\tequi)^{T})^{-1}$. In summary, we obtain that
\begin{equation} \label{eq:prt}
        \boldsymbol{p}_{R}(t) = C_{p}\boldsymbol{q}_{R}(t)  = X(\tequi) Q_{R}(\tequi)^{\dagger} \boldsymbol{q}_{R}(t),
    \end{equation}
    which is of the form~\eqref{eq:particularp}
    and thus contained in $\Span(X(\tequi))$, as desired.

    In order to analyze the convergence of $\boldsymbol{p}_{R}(t)$, we relate it to Chebyshev polynomial interpolation of $\bx$. The following lemma follows from classical approximation theory, see, e.g.,~\cite[Lemma~2.2]{Kressner2011}.
\begin{lemma}
    \label{thm:cheb}
            Let $\boldsymbol{q}_{R}(t) \in \R^{R+1}$ be defined as in~\eqref{eq:chebpolynomial},  containing the Chebyshev polynomials up to degree $R$. Under Assumption~\ref{assume:analyticity} there exists an approximation of the form 
            \[
            \boldsymbol{x}_{R}(t) = C_x \boldsymbol{q}_{R}(t), \quad C_x = \left[
            \bc_{0,x}, \bc_{1,x}, \cdots, \bc_{R,x} \right] \in \R^{n\times (R+1)},
            \]
such that 
        $
            \|\boldsymbol{c}_{k,x} \|_2 \leq 2 \kappa_{\rho}  \rho^{-k}
        $
        and
        \begin{equation*}
            \max_{ t \in \left[ -1,1\right]} \| \boldsymbol{x}_{R}(t) -\boldsymbol{x}(t) \|_2 \leq \frac{2\kappa_{\rho}}{\rho-1 } \rho^{-R}.
        \end{equation*}
    \end{lemma}
    Following the arguments in~\cite{Demanet2019a} for scalar functions, Lemma~\ref{thm:cheb} allows us to estimate the extrapolation error for $\boldsymbol{p}_{R}(t)$ if $R \sim \sqrt{M}$.
\begin{theorem}
    \label{thm:ls}
        Suppose that Assumption~\ref{assume:analyticity} holds and $R \leq \frac{1}{2}\sqrt{M-1}$.  Then the vector-valued polynomial $\boldsymbol{p}_{R} \in \PP_R$ defined in~\eqref{eq:prt} satisfies for every $t \in (1, (\rho + \rho^{-1})/2)$ the error bound
        \[
            \| \boldsymbol{x}(t) - \boldsymbol{p}_{R}(t)\|_2 \leq
            2\kappa_{\rho}   \Big[  \frac{1}{1-r}+ 
            \frac{C(M,R) \rho}{(\rho -1)\sqrt{\rho^2 r^2 -1}} \Big] r^{R+1},
        \]
        with $r = (t + \sqrt{t^{2} -1})/ \rho < 1$ and $C(M,R)$ defined as in Theorem~\ref{theorem:main}.
    \end{theorem}
    \begin{proof}
        Letting $\boldsymbol{x}_{R}$ be the polynomial from Lemma~\ref{thm:cheb}, we write
        \begin{align}
            \|\boldsymbol{x}(t) - \boldsymbol{p}_{R}(t) \|_2 & \leq    \|\boldsymbol{x}(t) - \boldsymbol{x}_{R}(t) \|_2 +  \|\boldsymbol{x}_{R}(t) - \boldsymbol{p}_{R}(t) \|_2 \nonumber \\
            & = \Big\| \sum_{k=R+1}^{\infty} \boldsymbol{c}_{k,x} q_{k}(t)\Big\|_2 + \|  (C_{x}- C_{p}) \boldsymbol{q}_{R}(t) \|_2 \nonumber \\
            & \leq   \sum_{k=R+1}^{\infty} \|  \boldsymbol{c}_{k,x} \|_2 | q_{k}(t)| + \|C_{x}- C_{p} \|_2 \| \boldsymbol{q}_{R}(t) \|_2.
            \label{eq:proof2}
        \end{align}
        To treat the second term in~\eqref{eq:proof2}, first note that, by definition, we have \[
X_R(\tequi) = \left[ \bx_R(t_{i-M})\,|\,\cdots\,|\,\bx_R(t_{i-1}) \right] = C_x Q_R(\tequi)         \]
and hence $C_x = X_R(\tequi) Q_R(\tequi)^\dagger$. Setting $\sigma:= \sigma_{\min}(Q_R(\tequi)) = 1/\|Q_R(\tequi)^\dagger\|_2$, we obtain
\begin{align*}
 \|C_x - C_p\|_2 & = \| (X_R(\tequi) - X(\tequi) ) Q_R(\tequi)^\dagger\|_2 
  \le \| X_R(\tequi) - X(\tequi)\|_2 / \sigma \\
 & \le \frac{\sqrt{M}}{\sigma} \cdot \max_{k=1,..,M} \| \boldsymbol{x}_{R}(t_{k}) - \boldsymbol{x}(t_{k}) \|_2 
 \leq  
    \frac{\sqrt{M}}{\sigma} \frac{2\kappa_{\rho}}{\rho-1}  \rho^{-R},
\end{align*}
where we used Lemma~\ref{thm:cheb} in the last inequality. Applying, once more, 
Lemma~\ref{thm:cheb} to the first term in~\eqref{eq:proof2} gives
\begin{align} \label{eq:proof3}
    \|\boldsymbol{x}(t) - \boldsymbol{p}_{R}(t) \|_2 & \leq   2\kappa_{\rho}   \Big[  \sum_{k=R+1}^{\infty} \rho^{-k} |q_{k}(t)| + 
    \frac{\sqrt{M}}{\sigma} \frac{\rho^{-R}}{\rho -1 } \| \boldsymbol{q}_{R}(t) \|_2  \Big]
\end{align} 
Because $|q_{k}(t)| \leq (t + \sqrt{t^{2}-1})^{k} \leq \rho^k r^k$ for $t>1$, we have that
\begin{align}
\| \boldsymbol{q}_{R}(t) \|_2^2 & \le \sum_{k = 0}^R (\rho r)^{2k} = (\rho r)^{2R}  \sum_{k = 0}^R (\rho r)^{-2k}   \le
\frac{(\rho r)^{2R+2}}{\rho^2 r^2-1}.
\label{eq:boundqr}
\end{align}
Inserted into~\eqref{eq:proof3}, this gives
        \begin{align*}
            \|\boldsymbol{x}(t) - \boldsymbol{p}_{R}(t) \|_2 & \leq   2\kappa_{\rho}   \Big[  \sum_{k=R+1}^{\infty} \rho^{-k} \rho^{k} r^{k} + 
            \frac{\sqrt{M}}{\sigma} \frac{\rho^{-R} (\rho r)^{R+1}}{(\rho -1)\sqrt{\rho^2 r^2 -1} } \Big] \\ & \leq   2\kappa_{\rho}   \Big[  \frac{1}{1-r}+ 
            \frac{\sqrt{M}\rho}{\sigma (\rho -1)\sqrt{\rho^2 r^2 -1}} \Big] r^{R+1}.
        \end{align*}
The proof is completed by inserting the lower bound 
    \begin{equation}
        \sigma = \sigma_{\min}(Q_{R}(\tequi)) \geq \frac{\sqrt{2}}{5\sqrt{5}} \frac{\sqrt{M-1}}{\sqrt{2R +1}},
        \label{eq:lb}
    \end{equation}
    which holds when $R \leq \frac{1}{2}\sqrt{M-1}$ according to~\cite[Theorem~4]{Demanet2019a}.
    \end{proof}

    Using Theorem~\ref{thm:ls} with $t=t_{i}$ and inserting the result in~\eqref{eq:boundres}, we have proven the statement of Theorem~\ref{thm:res}.

\subsection{Optimality of the prediction with compression}

When the matrix $X(\boldsymbol{t})$ is compressed via POD (Algorithm~\ref{alg:2}) or the randomized range finder (Algorithm~\ref{alg:4}), the orthonormal basis $Q \in \mathbb{R}^{n \times m}$ used in Algorithm~\ref{alg:1} spans a lower-dimensional subspace $\mathcal{S} \subseteq \Span(X)$.

 \begin{corollary} \label{corollary:compressed}
     Suppose that Algorithm~\ref{alg:1} is used with an orthonormal basis satisfying $\|(QQ^{T} -I) X(\boldsymbol{t_{\mathsf{equi}}}) \|_{2} \leq \varepsilon$ for some tolerance $\varepsilon > 0$. Under Assumption~\ref{assume:analyticity}, the initial guess 
     $\boldsymbol{s}^{*}$
     constructed by the algorithm satisfies the error bound
    \[
    \| A(t_{i})\boldsymbol{s}^{*} - \boldsymbol{b}(t_{i})\|_2  \leq 2 \|A(t_{i})\|_2 \kappa_{\rho}  \left[     \frac{1}{1-r}+ \frac{C(M,R) \rho} {\sqrt{\rho^2 r^2-1}} \left(\frac{1}{\rho -1}  + \frac{ \varepsilon  \rho^{R}}{2 \sqrt{M} \kappa_{\rho}} \right) \right] r^{R+1}
\]
    for any $R \leq \frac{1}{2} \sqrt{M-1}$. 
    \label{thm:res_cor}
\end{corollary}

\begin{proof}
Let $\boldsymbol{p}_{R}(t) = X(\tequi) Q_{R}(\tequi)^{\dagger} \boldsymbol{q}_{R}(t)$ be the polynomial constructed in~\eqref{eq:prt}. Using that $\boldsymbol{s}^{*}$ satisfies the minimization problem~\eqref{eq:minrecall} and   $QQ^T \boldsymbol{x}(t_{i}) \in  \mathcal{S} = \Span(Q)$, we obtain:
    \begin{align}
    \| A(t_i) \boldsymbol{s^{*}} - \boldsymbol{b} (t_i)  \|_2 
    &\leq   \| A(t_i) Q Q^{T} \boldsymbol{x} (t_i) - \boldsymbol{b} (t_i)  \|_2 \nonumber \\ &\leq \|A(t_{i})\| _2\big[ \|(QQ^{T}-I) (\boldsymbol{x} (t_i) -\boldsymbol{p}_{R}(t_i) ) \|_2 \nonumber \\ & \quad \quad \quad  \quad \quad + \|  (Q Q^{T}- I)\boldsymbol{p}_{R}(t_i)  \|_2 \big] \nonumber \\ &\leq  \|A(t_{i})\|_2 \big[ \|\boldsymbol{x} (t_i) -\boldsymbol{p}_{R}(t_i) \|_2 \nonumber \\ & \quad \quad \quad  \quad \quad + \|  (Q Q^{T}- I)X(\tequi) Q_{R}(\tequi)^{\dagger} \boldsymbol{q}_{R}(t_i)  \|_2 \big].
    \nonumber 
\end{align}
The first term is bounded using Theorem~\ref{thm:ls} with $t = t_i$. For the second term, we use the bound in~\eqref{eq:boundqr} on $\|\boldsymbol{q}_{R}(t_{M+1})\|$ to obtain
\begin{align*}
    \|  (Q Q^{T}- I)X(\tequi) Q_{R}(\tequi)^{\dagger} \boldsymbol{q}_{R}(t) \|_2 &\leq \|  (Q Q^{T}- I)X(\tequi) \|_{2} \|\boldsymbol{q}_{R}(t_{M+1})\|_2 / \sigma \\  &\leq  \frac{ \varepsilon (\rho r)^{R+1}}{\sigma \sqrt{\rho^2 r^2-1}},
\end{align*}
with $\sigma:= \sigma_{\min}(Q_R(\tequi))$. The proof is completed using 
the lower bound~\eqref{eq:lb} on x$\sigma$.
\end{proof}

\section{Numerical results: Test Case}
\label{sec:test}

To test the subspace acceleration algorithms proposed in Section~\ref{sec:algorithms}, we first consider a simplified setting, an elliptic PDE with an explicitly given time- and space-dependent coefficient $a(\boldsymbol{x},t)$ and source term $g(\boldsymbol{x},t)$:
\begin{align}
    \label{eq:diff}
    \begin{cases}
        \nabla \cdot (a(\boldsymbol{x},t) \nabla f(\boldsymbol{x},t)) = g(\boldsymbol{x},t)  & \quad \quad \text{in } \Omega \\
        f(\boldsymbol{x},t) = 0   &\quad \quad \text{on } \partial \Omega 
    \end{cases}
\end{align}
We consider the domain $\Omega = \left[0,1\right]^{2} \subset \mathbb{R}^{2}$ and discretize~\eqref{eq:diff} on a uniform two-dimensional Cartesian grid using a centered finite difference scheme of order $4$. This leads to a linear system for the vector of unknowns $\boldsymbol{{f}}(t)$, for which both the matrix and the right-hand side depend on $t$:
    \begin{equation} \label{eq:diff1}
           A(t) \boldsymbol{{f}}(t)= \boldsymbol{{g}}(t) .
    \end{equation}
We discretize the time variable on the interval $\left[t_{0}, \, t_{f} \right]$ with a uniform timestep $\Delta t$ on $N_{t}$ points, such that $t_{f} = t_{0} + N_{t} \Delta t$. Evaluating~\eqref{eq:diff1} in these $N_{t}$ instants, we obtain a  sequence of linear systems of the same type as~\eqref{eq:seq}.

We set $a(\boldsymbol{x},t) = \exp^{\left[-(x-0.5)^{2} - (y-0.5)^{2}\right]} \cos (tx) +2.1$ and choose the right-hand side $g(\boldsymbol{x},t)$  such that \[f(\boldsymbol{x},t) = \sin(4  \pi y t)  \sin(15 \pi x t) \left[ 1+ \sin(15  \pi x t) \cos(3  \pi y t) \exp^{\left[(x-0.5)^{2} + (y-0.5)^{2} -0.25^{2} \right]} \right] \] is the exact solution of~\eqref{eq:diff}.
 The tests are performed using MATLAB 2023a on an M1 MacbookPro. We employ GMRES as  iterative solver for the linear system, with tolerance $10^{-7}$ and incomplete LU factorization as preconditioner. We  start the simulations at $t_{0}= 2.3\, s$ and perform  $N_{t} = 200$ timesteps.
\begin{figure}[H]
    \centering
        \begin{subfigure}[t]{0.49\textwidth}
            \centering
            \includegraphics[width=\textwidth]{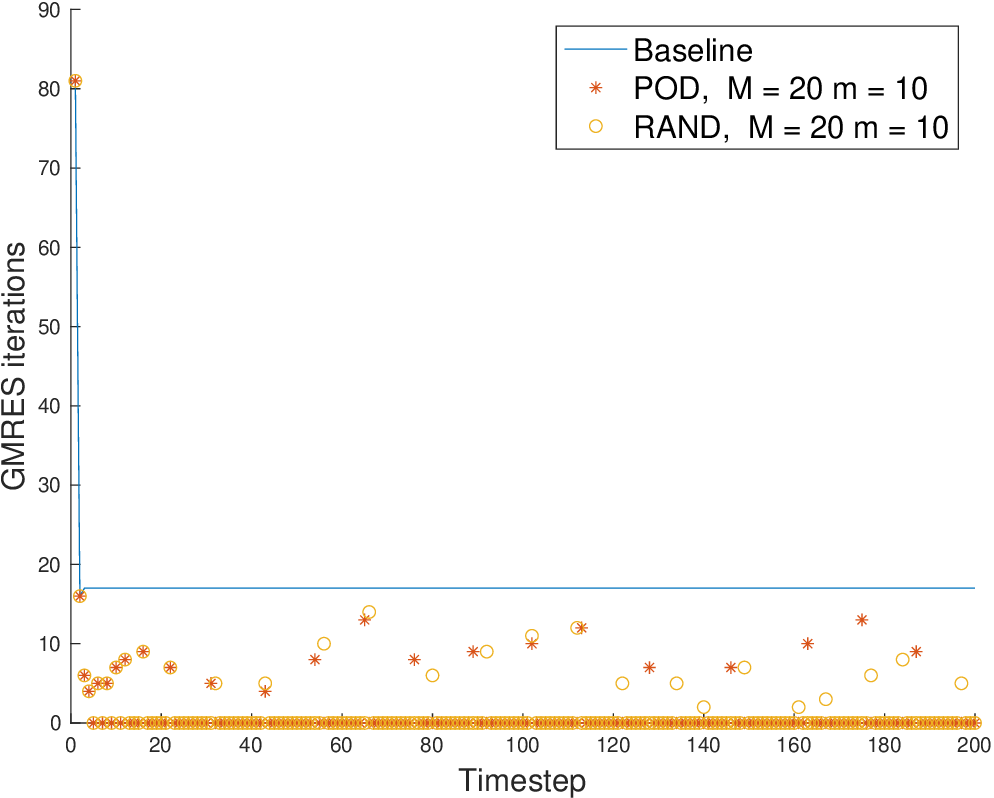}
            \subcaption{$\Delta t = 10^{-5}$}
            \label{fig:1_1}
        \end{subfigure}
        \hfill
        \begin{subfigure}[t]{0.49\textwidth}
            \centering
            \includegraphics[width=\textwidth]{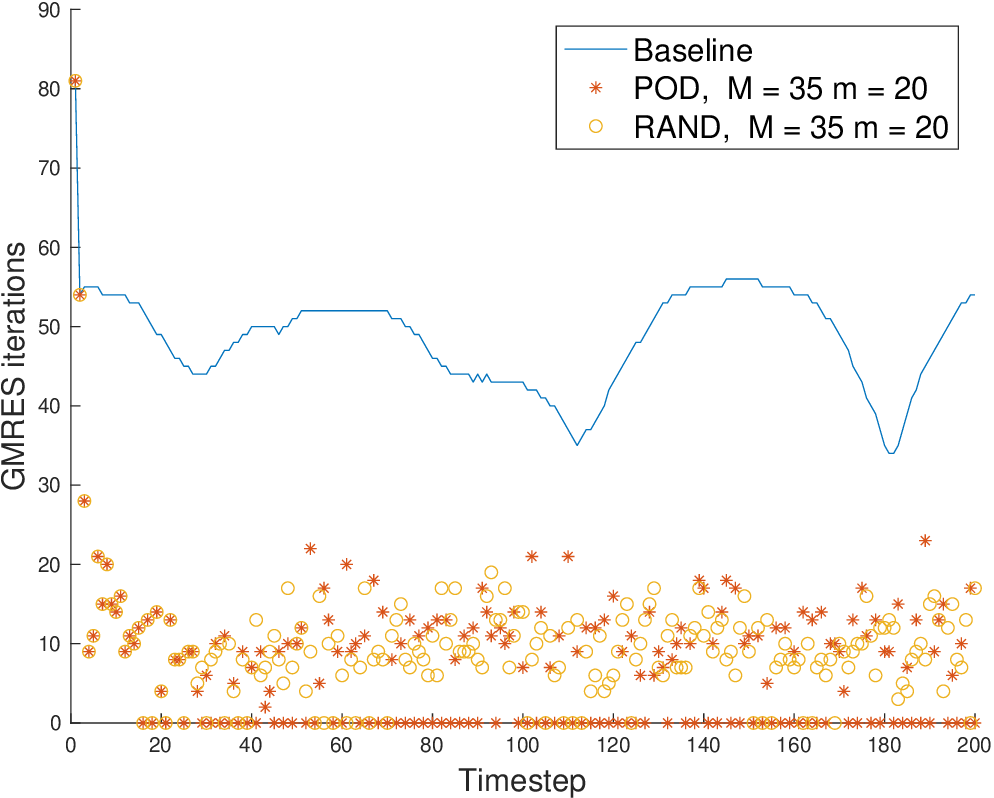}
            \subcaption{$\Delta t = 10^{-3}$}
            \label{fig:1_2}
        \end{subfigure} 
        \caption{\centering GMRES iterations per timestep when solving equation~\eqref{eq:diff1} with different initial guesses.}
        \label{fig:1}
\end{figure}

The results reported in Figure~\ref{fig:1} use a spatial grid of dimension $100 \times 100$, leading to linear systems of size $n = 10000$. Different values of $M$, the number of previous solutions retained in the history matrix $X$, and $m$, the dimension of the reduced-order model, were tested. We found that the choices $M = 20, \, m = 10$ and $M = 35,\, m = 20$ lead to good performance for $\Delta t = 10^{-5}$ and $\Delta t = 10^{-3}$, respectively.
The baseline is (preconditioned) GMRES with the previous solution used as initial guess; the resulting number of iterations is indicated with the solid blue line (\textit{``Baseline"}) in Figure~\ref{fig:1}. This is compared to the number of iterations obtained by applying GMRES when Algorithm~\ref{alg:1} is employed to compute the initial guess, in combination with both the POD basis in Algorithm~\ref{alg:2} (\textit{``POD"} in the graph) and the Randomized Range Finder in Algorithm~\ref{alg:4} (\textit{``RAND"} in the graph). For the Randomized Range Finder algorithm, the matrix $\Omega$ is computed from scratch only every $50$ timesteps, while in the other timesteps is updated as described in Algorithm~\ref{alg:4}, resulting in a computationally efficient version of the algorithm. Both the POD and Randomized versions of the acceleration method give a remarkable gain in computational time with respect to the baseline. 

When employing  $\Delta t = 10^{-5}$, in Figure~\ref{fig:1_1}, the number of iterations computed by the linear solver vanishes most of the time, since the initial residual computed with the new initial guess is already below the tolerance, set to $10^{-7}$ in this case. It is worth noticing that the new randomized method gives an acceleration comparable to the existing POD one, but it requires a much lower computational cost, as described in Section~\ref{sec:algorithms}. 

The results obtained for larger timesteps, in Figure~\ref{fig:1_2}, are slightly worse, as expected, since it is less easy to predict new solutions using the previous ones when they are further apart in time. Nevertheless, the gain of the acceleration method is still visible, obtaining always less than half iterations with respect to the baseline and adding the solution of a reduced-order system of dimension $m= 20 $ only, compared to the full solution of dimension $10000$. 
The resulting advantage of the new method can indeed be observed in  Figure~\ref{fig:5}, which compares the computational time needed by the solver using the baseline approach with the one obtained by using the new guess (this includes the time employed to compute the guess). The timings showed are the ones needed to produce the results in Figure~\ref{fig:1}. The time employed by the POD method has not been included since it is significantly higher than the baseline, as predicted by the analysis in Section~\ref{sec:POD}.

\begin{figure}[h]
    \centering
        \begin{subfigure}[t]{0.49\textwidth}
            \centering
            \includegraphics[width=\textwidth]{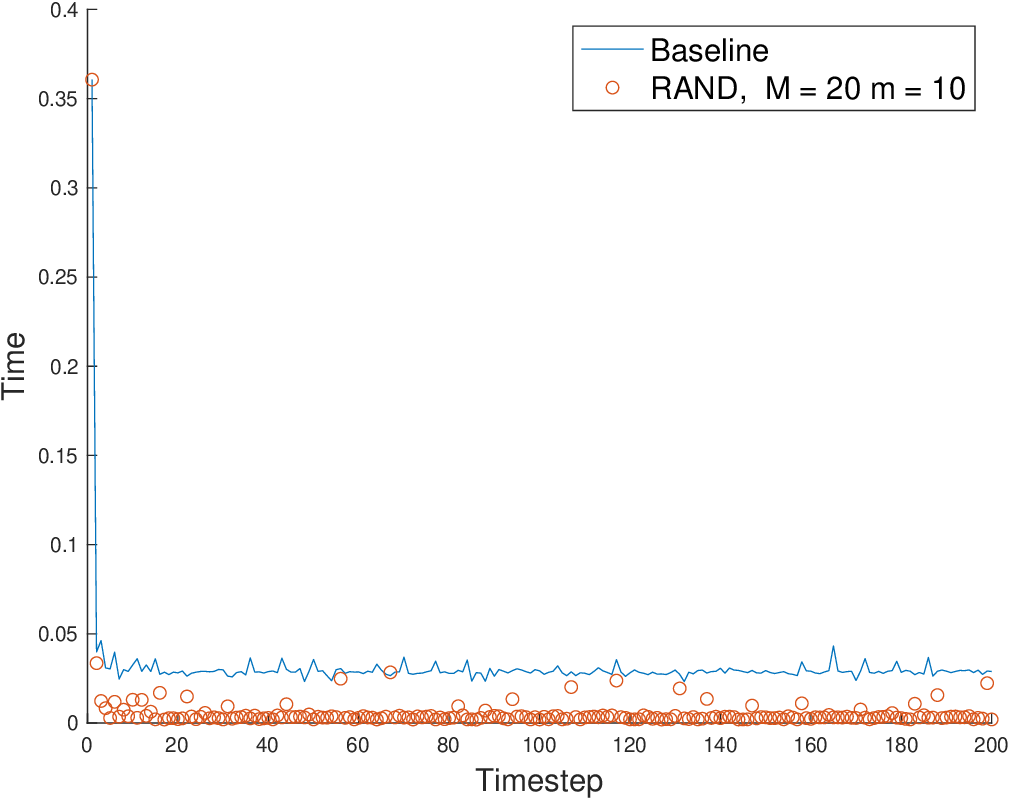}
            \subcaption{$\Delta t = 10^{-5} $ }
        \end{subfigure}
        \hfill
        \begin{subfigure}[t]{0.49\textwidth}
            \centering
            \includegraphics[width=\textwidth]{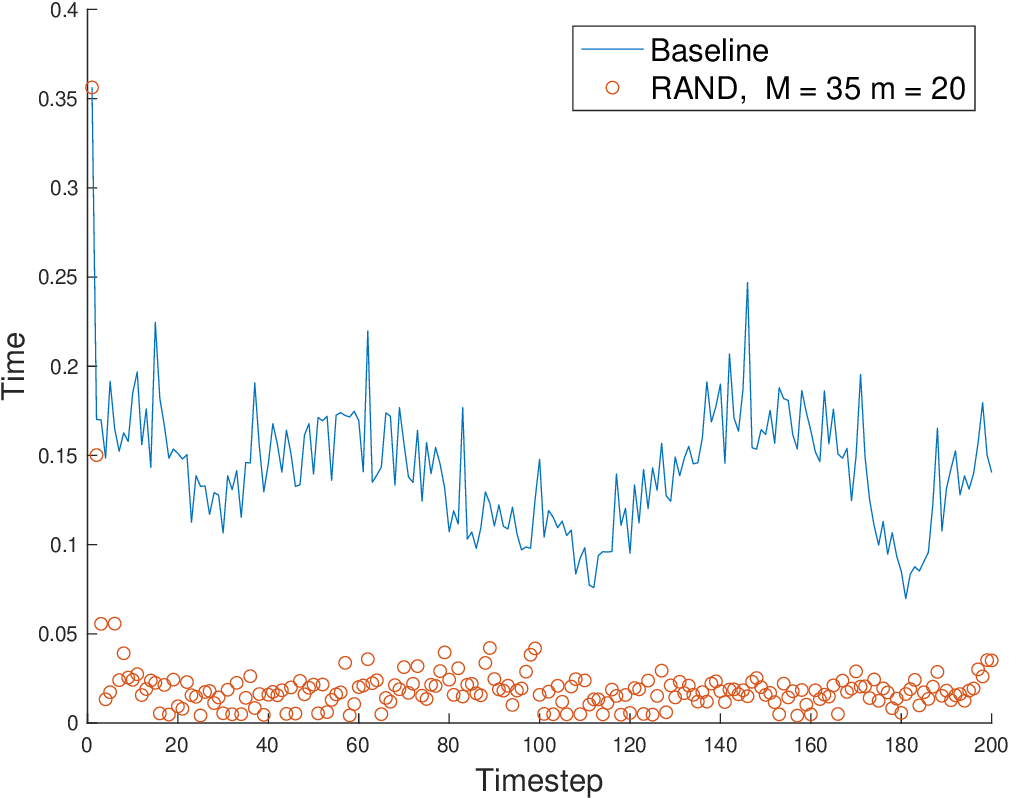}
            \subcaption{$\Delta t = 10^{-3} $}
        \end{subfigure} 
       
        \caption{  \centering Computational time per timestep corresponding to Figure~\ref{fig:1_1} and  Figure~\ref{fig:1_2}. The average speedup per iteration of the randomized method with respect to the baseline is a factor 9 for $\Delta t = 10^{-5} $ and a factor 10 for $\Delta t = 10^{-3}$. }
        \label{fig:5}
\end{figure}

\section{Numerical results: Plasma Simulations}
\label{sec:GBS}

In this Section, we apply the subspace acceleration method to the numerical simulation of plasma turbulence in the outermost plasma region of tokamaks, where the plasma enters in contact with the surrounding external solid walls, resulting in strongly non-linear  phenomena occurring on a large range of time and length scales. 
\begin{figure}[h]
    \centering
        \includegraphics[width=0.7\textwidth]{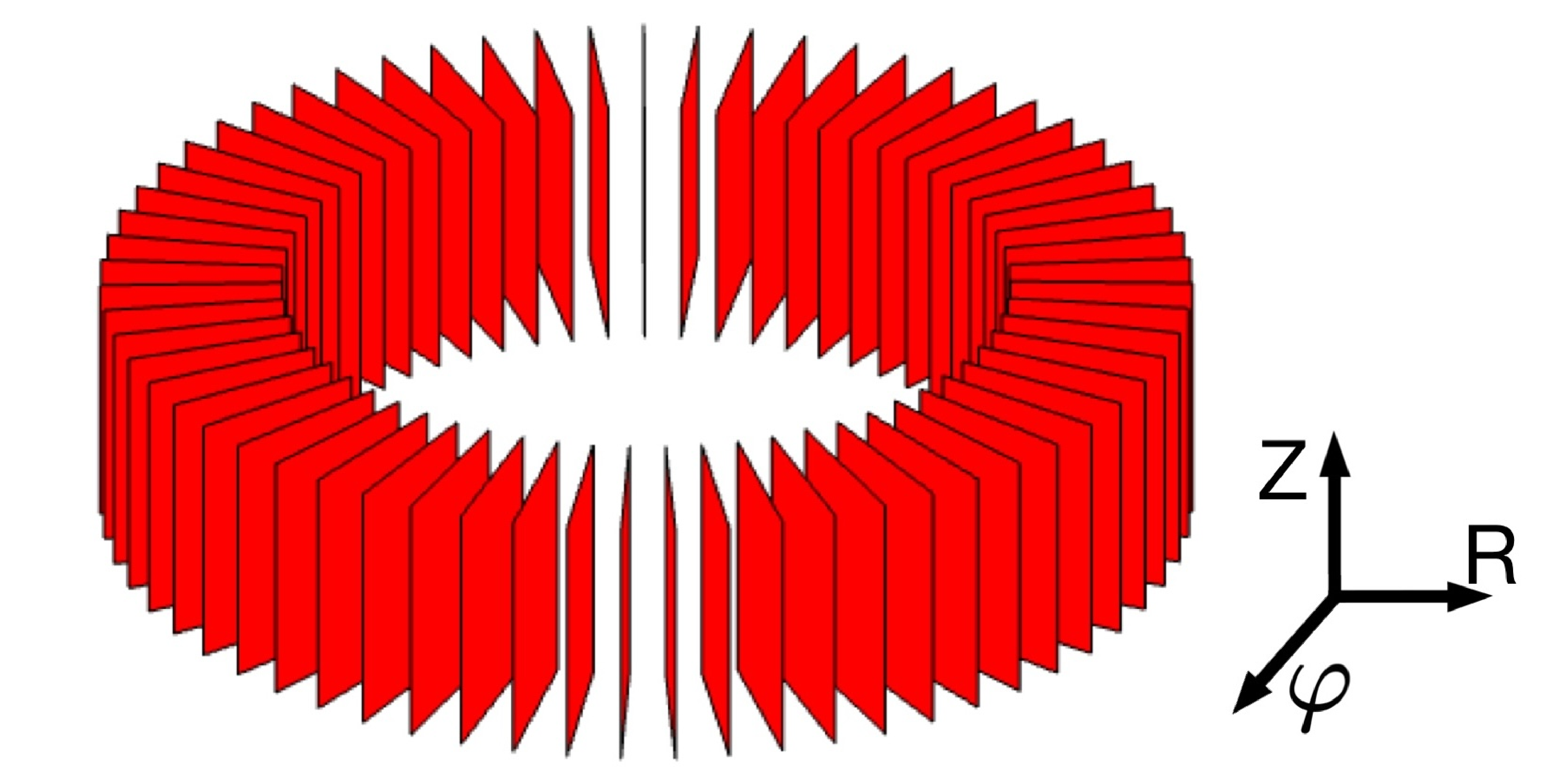}
    	\caption {\centering GBS computational domain. The toroidal direction is along $\varphi$, the radial direction is along $R$, and the vertical direction is along $Z$. The domain consists of $N_{\varphi}$ rectangular poloidal planes, each discretized on a $N_{R} \times N_{Z}$ Cartesian grid.}
    	\label{fig:tok}
\end{figure}
In this work, we consider GBS (Global Braginskii Solver) ~\cite{Giacomin2021,Ricci2012}, a three-dimensional, flux-driven, two-fluid code developed for the simulation of the plasma dynamics in the boundary of a fusion device.
GBS implements the Braginskii two-fluid model~\cite{Braginskii1965}, which describes a quasi-neutral plasma through the conservation of density, momentum, and energy. This results in six coupled three-dimensional time-evolving non-linear equations which evolve the plasma dynamics in $\Omega$, a 3D toroidal domain with rectangular poloidal cross section, as represented in Figure~\ref{fig:tok}.
The fluid equations are coupled with Maxwell equations, specifically Poisson and Ampére, elliptic equations for the electromagnetic variables of the plasma. In the limit considered here the elliptic equations reduce to a set of two-dimensional  algebraic constraints decoupled along the toroidal direction, therefore to be satisfied independently on each poloidal plane. The differential equations are spatially discretized  on a uniform Cartesian grid employing a finite difference method, resulting in a system of differential-algebraic equations of index one~\cite{Hairer1989}:
\begin{equation}
    \begin{cases}
        \partial_{t} \boldsymbol{f}(t)= \boldsymbol{\mathcal{Y}}(\boldsymbol{f}(t),\boldsymbol{x}(t)) \quad \quad \quad \, \, \, \text{in } \Omega
        \\ A_{k}(\boldsymbol{f}(t)) \boldsymbol{x}_{k}(t) =  \boldsymbol{b}_{k}(\boldsymbol{f}(t)) \quad \quad \text{for each $k$th poloidal plane }
    \end{cases}
    \label{eq:sys_semidisc}
\end{equation}
where $\boldsymbol{\mathcal{Y}}(\boldsymbol{f}(t),\boldsymbol{x}(t))$ is a non-linear, 6-dimensional differential operator and  
\begin{align*} &\boldsymbol{x}(t) = \left[ \boldsymbol{x}_{1}(t), \cdots, \boldsymbol{x}_{k}(t) \cdots, \boldsymbol{x}_{N_{Z}}(t)  \right] \in \mathbb{R}^{N_{R}N_{\varphi}N_{Z}}, \\ &\boldsymbol{f}(t) = \left[ \boldsymbol{f}_{1}(t), \cdots, \boldsymbol{f}_{k}(t) \cdots, \boldsymbol{f}_{N_{Z}}(t)  \right]  \in \mathbb{R}^{N_{R}N_{\varphi}N_{Z}}
\end{align*}
 are the vector of, respectively, the electromagnetic and fluid quantities solved for by GBS, where the solutions of all the $N_{Z}$ poloidal planes are stacked together. More precisely,  the time evolution of the fluid variables, $\boldsymbol{f}$, is  coupled with the set of linear systems $A_{k}(\boldsymbol{f}(t)) \boldsymbol{x}_{k}(t) =  \boldsymbol{b}_{k}(\boldsymbol{f}(t))$ which result from the discretization of Maxwell equations. Indeed, the matrix $A_{k} \in \mathbb{R}^{N_{R}N_{Z} \times N_{R}N_{Z} }$ and right-hand side $\boldsymbol{b_{k}} \in \mathbb{R}^{N_{R}N_{Z}}$ depend on time through $\boldsymbol{f}$.

In GBS, system~\eqref{eq:sys_semidisc} is integrated using a Runge-Kutta scheme of order four,  on the discrete times $\left\{ t_{i}  \right \}_{i=1}^{N_{t}}$, with step-size $\Delta t$. Given $\boldsymbol{f}^{i}$ and $\boldsymbol{x}^{i}$, the value of $\boldsymbol{f}$ and $\boldsymbol{x}$ at time $t_{i}$,  the computation of $\boldsymbol{ {f}}^{i+1}$, requires performing three intermediate substeps where the quantities $\boldsymbol{ {f}}^{i+1, j}$ for $ j =1,2,3$ are computed.
To guarantee the consistency and convergence of the Runge-Kutta integration method~\cite{Hairer1989}, the algebraic constraints are solved at every substep, computing  $\boldsymbol{ {x}}^{i+1,j}_{k}$ for $ j =1,2,3$ and for each $k-$th poloidal plane.
As a consequence, the linear systems  $A_{k}(\boldsymbol{f}(t)) \boldsymbol{x}_{k}(t) =  \boldsymbol{b}_{k}(\boldsymbol{f}(t))$  are assembled and solved four times for each of the $N_{\varphi}$ poloidal planes, to advance the full system~\eqref{eq:sys_semidisc} by one timestep. Since the timestep $\Delta t$ is constrained to be small from the stiff nature of the GBS model, the solution of the linear systems is among the most computationally expensive part of GBS simulations.

In GBS, the linear system is solved using GMRES, with the algebraic multigrid preconditioner \textit{boomerAMG} from the HYPRE library~\cite{Falgout2002}, a choice motivated by previous investigations~\cite{Giacomin2021}. 
The subspace acceleration algorithm proposed in Section~\ref{sec:algorithms} is implemented in the GBS code and, given the results shown in Section~\ref{sec:test}, the randomized version of the algorithm is chosen. The results reported are obtained from GBS simulations on one computing node. The poloidal planes of the computational domain are distributed among 16 cores, specifically of type Intel(R) Core i7-10700F CPU at 2.90GHz. GBS is implemented in Fortran 90, and relies on the PETSc library~\cite{Balay1998} for the linear solver and  Intel MPI 19.1 for the parallelization.

We consider the simulation setting described in~\cite{Giacomin2021}, taking as initial conditions the results of a simulation in a turbulent state. 
We use a  Cartesian grid of size of $N_{R} = 150,\, N_{Z} = 300$ and  $N_{\varphi}=64$, with additional 4 ghost points in the $Z$ and $R$ directions. Therefore, the imposed algebraic constraints result in 64 sequences of linear systems of dimension $ N_{R} N_{Z} \times  N_{R} N_{Z} = 46816 \times 46816$. The timestep employed is $\Delta t = 0.7 \times 10^{-5}$. The sequence of linear systems we consider represents the solution of the Poisson equation on one fixed poloidal plane, but the same considerations apply to the discretization of Ampére equation.

\begin{figure}[H]
    \centering
        \begin{subfigure}[t]{0.49\textwidth}
            \centering
            \includegraphics[width=\textwidth]{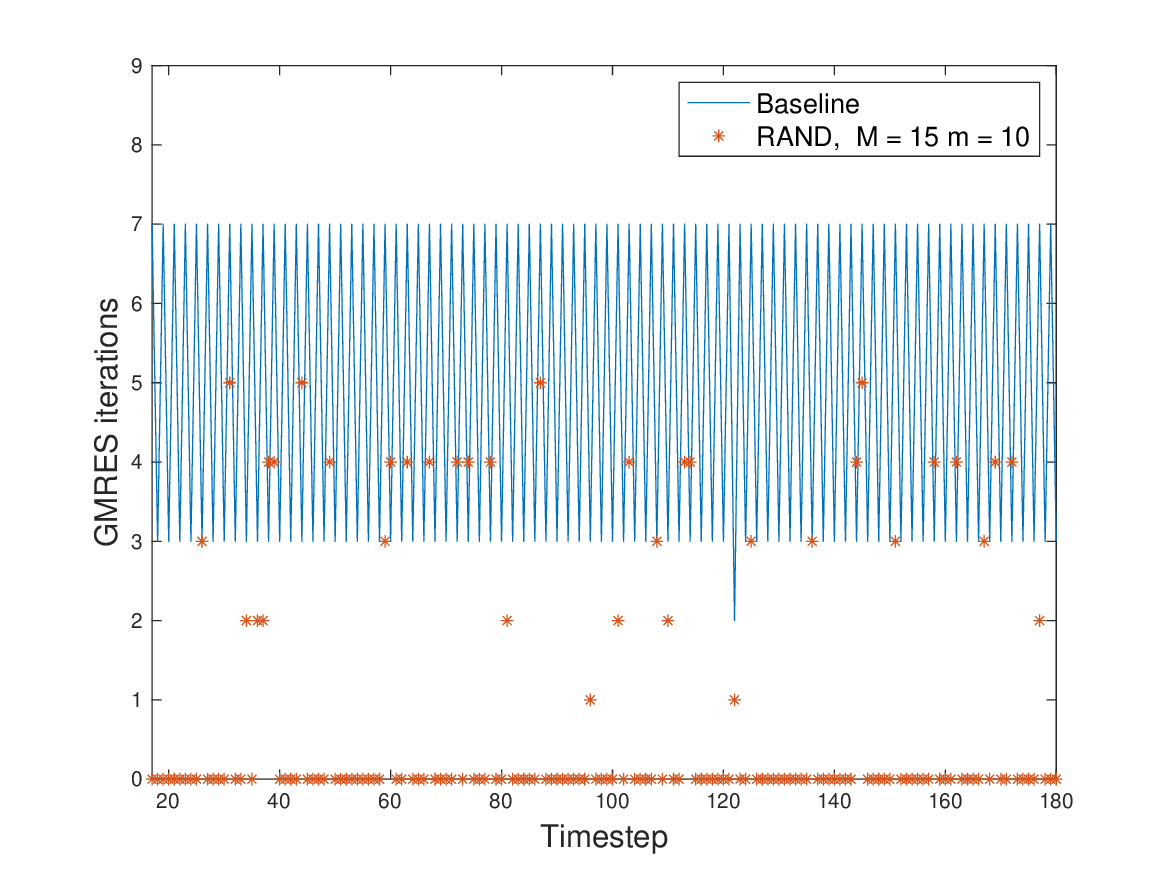}
            \subcaption{GMRES iterations per timestep}
            \label{fig:2_1}
        \end{subfigure}
        \hfill
        \begin{subfigure}[t]{0.49\textwidth}
            \centering
            \includegraphics[width=\textwidth]{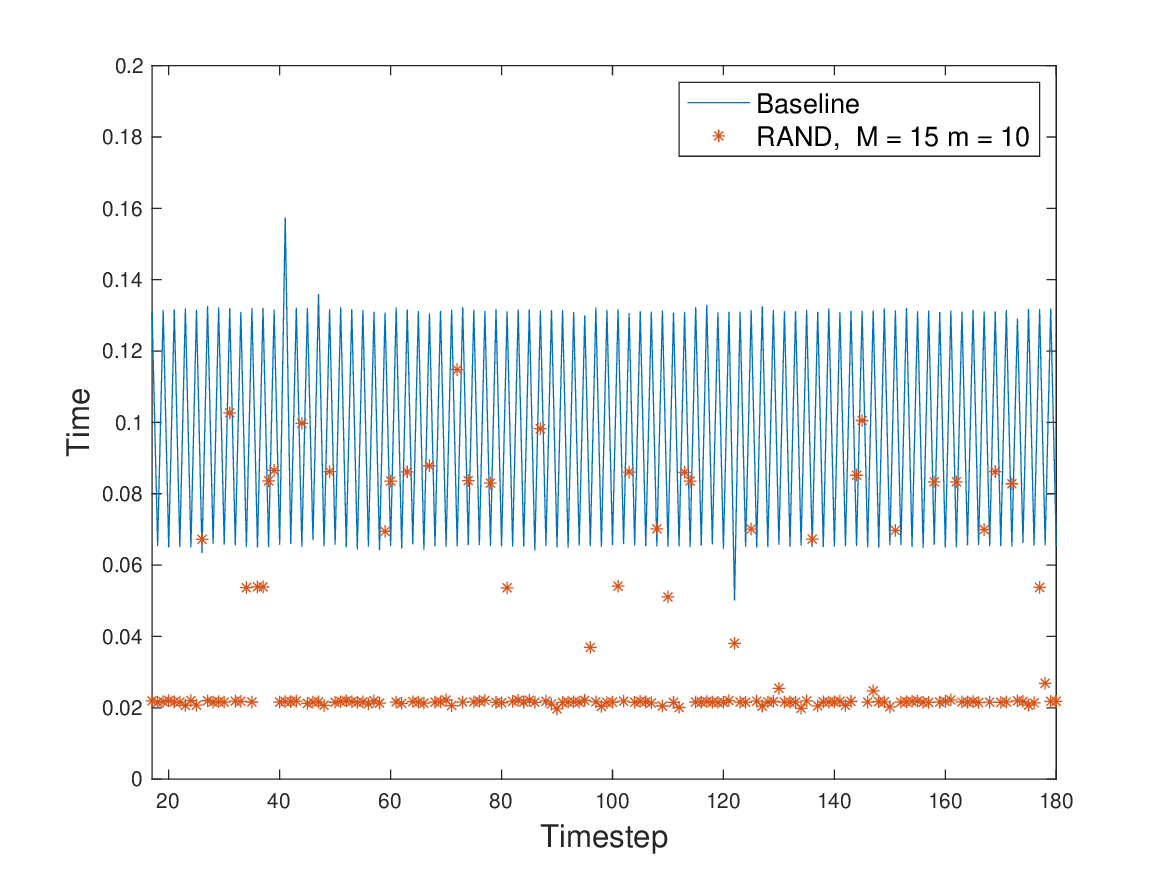}
            \subcaption{Computational time per timestep}
            \label{fig:2_2}
        \end{subfigure} 
        
        \caption{ \centering  Performance of the algorithm applied to the solution of Poisson equation in GBS simulations. The time  for the RAND algorithm is on average approximately one fourth of the time for the baseline.}
        \label{fig:2}
\end{figure}

In Figure~\ref{fig:2_1} the number of iterations obtained with the method proposed in Section~\ref{sec:algorithms}, denoted as \textit{``RAND"} is compared with the ones obtained using  the previous step solution as initial guess, depicted in blue as \textit{``Baseline"}. 
We notice that, employing the acceleration method, the number of GMRES iterations needed for each solution of the linear system is reduced by a factor 2.9, on average, at the cost of  computing a solution of an $m \times m $ reduced-order system.
In Figure~\ref{fig:2_2} the wall clock time required for the solution of the systems is shown. The baseline approach is compared to the accelerated method, where we also take into account the cost of computing the initial guess. Thanks to the randomized method employed, the process of generating the guess is fast enough to provide a time speed up of a factor of 6.5 per iteration.

The employed values of $M=15$, the number of previous solutions retained, and $m=10$, the dimension of the reduced-order model, are the ones found to give a good balance between the decrease in the number of iterations and the computational cost of the reduced-order model. In Table~\ref{tab:1} the results for different values of $M$ and $m$ are reported. It is worth noticing that an average  number of GMRES iterations per timestep smaller than one implies that often the initial residual
obtained with the initial guess is below the tolerance set for the solver.
It is possible to notice that higher values of $m$ lead to very small number of iterations, but the overall time speedup is reduced since the computation of the guess becomes more expensive.

\begin{table}[h]
    \begin{tabular}{c|c|c|c|c|c}
        M  & m  & \begin{tabular}[c]{@{}c@{}}Average time \\ per timestep {[}s{]}\end{tabular} & \begin{tabular}[c]{@{}c@{}}Time \\ speedup\end{tabular} & \begin{tabular}[c]{@{}c@{}}Average GMRES\\  iterations per timestep\end{tabular} & \begin{tabular}[c]{@{}c@{}}Iterations \\ speedup\end{tabular} \\ \hline
        15 & 6  & 0.0452        & 2.1797                                                  & 2.0183            & 2.4743                                              \\ \hline
        15 & 8  & 0.0347        & 2.8352                                                  & 1.1098            & 4.5            \\ \hline
        \textbf{15} & \textbf{10} & \textbf{0.0339}        & \textbf{2.9071}                                                  & \textbf{0.76}              & \textbf{6.552}          \\ \hline
        20 & 10 & 0.0358        & 2.7468                                                  & 0.7862            & 6.336          \\ \hline
        20 & 15 & 0.0435                                                                       & 2.2572                                                  & 0.5031                                                                           & 9.9                                                           \\ \hline
        30 & 8  & 0.0485        & 2.0255                                                  & 1.698             & 2.9328         \\ \hline
        30 & 12 & 0.0375        & 2.6185                                                  & 0.3758            & 13.25          \\ \hline
        30 & 15 & 0.05                                                                         & 1.9633                                                  & 0.579                                                                            & 8.3371                        
        \end{tabular}
    \caption{  \centering
        GBS simulations result corresponding to different values of $M$ and $m$. The iteration and time speedups are computed on the total of 180 linear systems, with respect to the baseline, that has an average number of 5 GMRES iterations and an average time per timestep of $0.0828$ s. The highlighted row corresponds to the best result obtained in terms of time speedup.}
    \label{tab:1}
    \end{table}

\section {Conclusions}

In this paper, we propose a novel approach for accelerating the solution of a sequence of large-scale linear systems that arises from, e.g., the discretization of time-dependent PDEs. Our method generates an initial guess from the solution of a reduced-order model, obtained by extracting relevant components of previously computed solutions using dimensionality reduction techniques. Starting from an existing POD-like approach, we  accelerate the process by employing a randomized algorithm.
A convergence analysis is performed, which applies to both approaches, POD and the randomized algorithm and shows how the accuracy of the method increases with the history size.  
A test case displays how POD leads to a noticeable decrease in the number of iterations, but at the same time a nearly equal decrease is achieved by the cheaper randomized method, that leads to a time  speedup per iteration of a factor 9.
In real applications such as the plasma simulations described in Section~\ref{sec:GBS}, the speedup is more modest, given the stiff nature of the problem  which constrains the timestep of the explicit integration method to be very small, but still practically  relevant. 

\vspace{3mm}

\textbf{Acknowledgements.} The authors thank the anonymous reviewers for helpful feedback.

This work has been carried out within the framework of the EUROfusion Consortium, via the Euratom Research and Training Programme (Grant Agreement No 101052200 — EUROfusion) and funded by the Swiss State Secretariat for Education, Research and Innovation (SERI). Views and opinions expressed are however those of the author(s) only and do not necessarily reflect those of the European Union, the European Commission, or SERI. Neither the European Union nor the European Commission nor SERI can be held responsible for them.

\textbf{Data Availability Statement.} The data that support the findings of this study are available upon reasonable request from the authors.

\printbibliography

@book{Balay1998,
author = {Balay, Satish and Gropp, William and McInnes, Lois Curfman and Smith, Barry F},
number = {17},
publisher = {Argonne National Laboratory},
title = {{PETSc, the portable, extensible toolkit for scientific computation}},
volume = {2},
year = {1998}
}

@article{Ricci2012,
author = {Ricci, P. and Halpern, F. D. and Jolliet, S. and Loizu, J. and Mosetto, A. and Fasoli, A. and Furno, I. and Theiler, C.},
doi = {10.1088/0741-3335/54/12/124047},
issn = {07413335},
journal = {Plasma Phys. Control. Fusion},
number = {12},
title = {{Simulation of plasma turbulence in scrape-off layer conditions: The GBS code, simulation results and code validation}},
volume = {54},
year = {2012}
}

@article{Soodhalter2014,
archivePrefix = {arXiv},
arxivId = {1301.2650},
author = {Soodhalter, Kirk M. and Szyld, Daniel B. and Xue, Fei},
doi = {10.1016/j.apnum.2014.02.006},
eprint = {1301.2650},
issn = {01689274},
journal = {Appl. Numer. Math.},
pages = {105--118},
title = {{Krylov subspace recycling for sequences of shifted linear systems}},
volume = {81},
year = {2014}
}

@article{Austin2021,
author = {Austin, Anthony P. and Chalmers, Noel and Warburton, Tim},
doi = {10.1137/20M1368677},
issn = {10957197},
journal = {SIAM J. Sci. Comput.},
number = {4},
pages = {C259--C289},
title = {{Initial guesses for sequences of linear systems in a GPU-accelerated incompressible flow solver}},
volume = {43},
year = {2021}
}

@article {Mathematics2012,
    AUTHOR = {De Sturler, Eric},
     TITLE = {Truncation strategies for optimal {K}rylov subspace methods},
   JOURNAL = {SIAM J. Numer. Anal.},
  FJOURNAL = {SIAM Journal on Numerical Analysis},
    VOLUME = {36},
      YEAR = {1999},
    NUMBER = {3},
     PAGES = {864--889},
      ISSN = {0036-1429},
   MRCLASS = {65F10 (15A18 65N22)},
  MRNUMBER = {1681025},
       DOI = {10.1137/S0036142997315950},
       URL = {https://doi.org/10.1137/S0036142997315950},
}

@article{Grinberg2011,
author = {Grinberg, Leopold and {Em Karniadakis}, George},
doi = {10.4208/cicp.301109.080410s},
issn = {19917120},
journal = {Commun. Comput. Phys.},
number = {3},
pages = {607--626},
title = {{Extrapolation-based acceleration of iterative solvers: Application to simulation of 3D flows}},
volume = {9},
year = {2011}
}

@article{Hestenes1952,
author = {Hestenes, M.R. and Stiefel, E.},
doi = {10.6028/jres.049.044},
issn = {0091-0635},
journal = {J. Res. Natl. Bur. Stand. (1934).},
number = {6},
pages = {409},
title = {{Methods of conjugate gradients for solving linear systems}},
volume = {49},
year = {1952}
}

@article{Giacomin2021,
title = {The GBS code for the self-consistent simulation of plasma turbulence and kinetic neutral dynamics in the tokamak boundary},
journal = {J. Comput. Phys.},
volume = {463},
pages = {111294},
year = {2022},
issn = {0021-9991},
doi = {https://doi.org/10.1016/j.jcp.2022.111294},
url = {https://www.sciencedirect.com/science/article/pii/S0021999122003564},
author = {M. Giacomin and P. Ricci and A. Coroado and G. Fourestey and D. Galassi and E. Lanti and D. Mancini and N. Richart and L.N. Stenger and N. Varini}
}

@article{Braginskii1965,
author = {Braginskii, S. I.},
journal = {Rev. Plasma Phys.},
month = {jan},
pages = {205},
title = {{Transport Processes in a Plasma}},
volume = {1},
year = {1965}
}

@article{Halko2011,    AUTHOR = {Halko, N. and Martinsson, P. G. and Tropp, J. A.},
     TITLE = {Finding structure with randomness: probabilistic algorithms
              for constructing approximate matrix decompositions},
   JOURNAL = {SIAM Rev.},
  FJOURNAL = {SIAM Review},
    VOLUME = {53},
      YEAR = {2011},
    NUMBER = {2},
     PAGES = {217--288},
      ISSN = {0036-1445},
   MRCLASS = {65F30 (60B20 68W20)},
  MRNUMBER = {2806637},
MRREVIEWER = {Thomas K. Huckle},
       DOI = {10.1137/090771806},
       URL = {https://doi.org/10.1137/090771806},
}

@article{Kressner2011,
    AUTHOR = {Kressner, Daniel and Tobler, Christine},
     TITLE = {Low-rank tensor {K}rylov subspace methods for parametrized
              linear systems},
   JOURNAL = {SIAM J. Matrix Anal. Appl.},
  FJOURNAL = {SIAM Journal on Matrix Analysis and Applications},
    VOLUME = {32},
      YEAR = {2011},
    NUMBER = {4},
     PAGES = {1288--1316},
      ISSN = {0895-4798},
   MRCLASS = {65F10 (15A69)},
  MRNUMBER = {2854614},
MRREVIEWER = {Christian Wieners},
       DOI = {10.1137/100799010},
       URL = {https://doi.org/10.1137/100799010},
}

@article{Saad1986,
author = {Saad, Youcef and Schultz, Martin H.},
doi = {10.1137/0907058},
issn = {0196-5204},
journal = {SIAM J. Sci. Stat. Comput.},
number = {3},
pages = {856--869},
title = {{GMRES: A Generalized Minimal Residual Algorithm for Solving Nonsymmetric Linear Systems}},
volume = {7},
year = {1986}
}

@article{Fischer1998,
author = {Fischer, Paul F.},
doi = {10.1016/S0045-7825(98)00012-7},
issn = {00457825},
journal = {Comput. Methods Appl. Mech. Eng.},
number = {1-4},
pages = {193--204},
title = {{Projection techniques for iterative solution of Ax = b with successive right-hand sides}},
volume = {163},
year = {1998}
}

@article{Markovinovic2006,
author = {Markovinovi{\'{c}}, R. and Jansen, J. D.},
doi = {10.1002/nme.1721},
issn = {00295981},
journal = {Int. J. Numer. Methods Eng.},
number = {5},
pages = {525--541},
title = {{Accelerating iterative solution methods using reduced-order models as solution predictors}},
volume = {68},
year = {2006}
}

@article{Tromeur-Dervout2006,
author = {Tromeur-Dervout, Damien and Vassilevski, Yuri},
doi = {10.1016/j.jcp.2006.03.014},
issn = {10902716},
journal = {J. Comput. Phys.},
number = {1},
pages = {210--227},
title = {{Choice of initial guess in iterative solution of series of systems arising in fluid flow simulations}},
volume = {219},
year = {2006}
}

@article{Soodhalter2020,
    AUTHOR = {Soodhalter, Kirk M. and de Sturler, Eric and Kilmer, Misha E.},
     TITLE = {A survey of subspace recycling iterative methods},
   JOURNAL = {GAMM-Mitt.},
  FJOURNAL = {GAMM-Mitteilungen},
    VOLUME = {43},
      YEAR = {2020},
    NUMBER = {4},
     PAGES = {e202000016, 29},
      ISSN = {0936-7195},
   MRCLASS = {65F10},
  MRNUMBER = {4175151},
       DOI = {10.1002/gamm.202000016},
       URL = {https://doi.org/10.1002/gamm.202000016},
}

@book{Golub2013,
author = {Golub, Gene H. and {Van Loan}, Charles F.},
edition = {Fourth},
isbn = {978-1-4214-0794-4; 1-4214-0794-9; 978-1-4214-0859-0},
publisher = {Johns Hopkins University Press, Baltimore, MD},
series = {Johns Hopkins Studies in the Mathematical Sciences},
title = {{Matrix computations}},
year = {2013}
}

@article{Brand2006,
    AUTHOR = {Brand, Matthew},
     TITLE = {Fast low-rank modifications of the thin singular value
              decomposition},
   JOURNAL = {Linear Algebra Appl.},
  FJOURNAL = {Linear Algebra and its Applications},
    VOLUME = {415},
      YEAR = {2006},
    NUMBER = {1},
     PAGES = {20--30},
      ISSN = {0024-3795},
   MRCLASS = {15A18},
  MRNUMBER = {2214744},
MRREVIEWER = {Ronald L. Smith},
       DOI = {10.1016/j.laa.2005.07.021},
       URL = {https://doi.org/10.1016/j.laa.2005.07.021},
}

@article{Demanet2019a,    AUTHOR = {Demanet, Laurent and Townsend, Alex},
     TITLE = {Stable extrapolation of analytic functions},
   JOURNAL = {Found. Comput. Math.},
  FJOURNAL = {Foundations of Computational Mathematics. The Journal of the
              Society for the Foundations of Computational Mathematics},
    VOLUME = {19},
      YEAR = {2019},
    NUMBER = {2},
     PAGES = {297--331},
      ISSN = {1615-3375},
   MRCLASS = {41A10 (30E10 65D05)},
  MRNUMBER = {3937956},
MRREVIEWER = {Yi Li},
       DOI = {10.1007/s10208-018-9384-1},
       URL = {https://doi.org/10.1007/s10208-018-9384-1},
}

@article{Morgan2000,
author = {Morgan, Ronald B.},
doi = {10.1137/S0895479897321362},
issn = {08954798},
journal = {SIAM J. Matrix Anal. Appl.},
number = {4},
pages = {1112--1135},
title = {{Implicitly restarted GMRES and Arnoldi methods for nonsymmetric systems of equations}},
volume = {21},
year = {2000}
}

@article{Volkwein2013,
author = {Volkwein, Stefan},
journal = {Lect. Notes, Univ. Konstanz},
pages = {4},
title = {{Proper orthogonal decomposition: Theory and reduced-order modelling}},
url = {http://topmath.ma.tum.de/foswiki/pub/IGDK1754/SummerSchool2013Data/volkwein-slides-1.pdf},
volume = {4},
year = {2013}
}

@article{Kunisch1999,
author = {Kunisch, K. and Volkwein, S.},
doi = {10.1023/A:1021732508059},
issn = {00223239},
journal = {J. Optim. Theory Appl.},
number = {2},
pages = {345--371},
title = {{Control of the Burgers equation by a reduced-order approach using proper orthogonal decomposition}},
volume = {102},
year = {1999}
}

@inproceedings{Falgout2002,
address = {Berlin, Heidelberg},
author = {Falgout, Robert D and Yang, Ulrike Meier},
booktitle = {Comput. Sci. --- ICCS 2002},
editor = {Sloot, Peter M A and Hoekstra, Alfons G and Tan, C J Kenneth and Dongarra, Jack J},
isbn = {978-3-540-47789-1},
pages = {632--641},
publisher = {Springer Berlin Heidelberg},
title = {{hypre: A Library of High Performance Preconditioners}},
year = {2002}
}

@article{Ye2020,
author = {Ye, Shuai and Lin, Yufei and Xu, Liyang and Wu, Jiaming},
doi = {10.3390/math8010119},
issn = {22277390},
journal = {Mathematics},
number = {1},
pages = {1--20},
title = {{Improving initial guess for the iterative solution of linear equation systems in incompressible flow}},
volume = {8},
year = {2020}
}

@article{Parks2006a,
    AUTHOR = {Parks, Michael L. and de Sturler, Eric and Mackey, Greg and
              Johnson, Duane D. and Maiti, Spandan},
     TITLE = {Recycling {K}rylov subspaces for sequences of linear systems},
   JOURNAL = {SIAM J. Sci. Comput.},
  FJOURNAL = {SIAM Journal on Scientific Computing},
    VOLUME = {28},
      YEAR = {2006},
    NUMBER = {5},
     PAGES = {1651--1674},
      ISSN = {1064-8275},
   MRCLASS = {65F10},
  MRNUMBER = {2272183},
MRREVIEWER = {Petr Tich\'{y}},
       DOI = {10.1137/040607277},
       URL = {https://doi.org/10.1137/040607277},
}

@book{Hairer1989,
    AUTHOR = {Hairer, Ernst and Lubich, Christian and Roche, Michel},
     TITLE = {The numerical solution of differential-algebraic systems by
              {R}unge-{K}utta methods},
    SERIES = {Lecture Notes in Mathematics},
    VOLUME = {1409},
 PUBLISHER = {Springer-Verlag, Berlin},
      YEAR = {1989},
     PAGES = {viii+139},
      ISBN = {3-540-51860-6},
   MRCLASS = {65L05},
  MRNUMBER = {1027594},
MRREVIEWER = {Rudolf Scherer},
       DOI = {10.1007/BFb0093947},
       URL = {https://doi.org/10.1007/BFb0093947},
}

@article{Fasoli2016,
author = {Fasoli, A. and Brunner, S. and Cooper, W. A. and Graves, J. P. and Ricci, P. and Sauter, O. and Villard, L.},
doi = {10.1038/NPHYS3744},
issn = {17452481},
journal = {Nat. Phys.},
number = {5},
pages = {411--423},
title = {{Computational challenges in magnetic-confinement fusion physics}},
volume = {12},
year = {2016}
}

@misc{Chen2023,
      title={An Incremental SVD Method for Non-Fickian Flows in Porous Media: Addressing Storage and Computational Challenges}, 
      author={Gang Chen and Yangwen Zhang and Dujin Zuo},
      year={2023},
      eprint={2308.15409},
      archivePrefix={arXiv},
      primaryClass={math.NA}
}

@article{Carlberg2016,
author = {Carlberg, Kevin and Forstall, Virginia and Tuminaro, Ray},
title = {Krylov-Subspace Recycling via the POD-Augmented Conjugate-Gradient Method},
journal = {SIAM Journal on Matrix Analysis and Applications},
volume = {37},
number = {3},
pages = {1304-1336},
year = {2016},
doi = {10.1137/16M1057693}
}

@article{Soodhalter2016,
author = {Soodhalter, Kirk M.},
title = {Block Krylov Subspace Recycling for Shifted Systems with Unrelated Right-Hand Sides},
journal = {SIAM Journal on Scientific Computing},
volume = {38},
number = {1},
pages = {A302-A324},
year = {2016},
doi = {10.1137/140998214}
}

\end{document}